\documentclass[12pt, reqno]{amsart}
\usepackage{amssymb, mathrsfs}
\usepackage{latexsym}
\usepackage{amsmath}
\usepackage{amsthm}
\usepackage{amsfonts}
\usepackage{dsfont, upgreek}
\usepackage{mathtools}
\usepackage{epsfig}
\usepackage{amscd}
\usepackage{graphicx}
\usepackage{tikz-cd}
\usepackage{enumitem}
\setlist[enumerate]{itemsep=0.5ex}

\usepackage{stmaryrd}

\usepackage{color}
\usepackage{tikz}
\usetikzlibrary{patterns}

\usepackage[left=1.4in,right=1.4in,top=1.2in,bottom=1.2in]{geometry}

\usepackage{tikz}
\usetikzlibrary{decorations.markings}
\usetikzlibrary{arrows.meta}

\usepackage{extarrows}

\usepackage{adjustbox}
\usepackage{pgfplots}
\usepgfplotslibrary{colormaps}
\usepackage{slashed}

\usepackage[colorlinks=true,linkcolor=blue,citecolor=blue]{hyperref}

\usepackage[colorinlistoftodos,prependcaption,textsize=tiny]{todonotes}  

\usepackage[all,cmtip]{xy}
\theoremstyle{plain}
\newtheorem{theorem}{Theorem}[section]

\newtheorem{lemma}[theorem]{Lemma}

\theoremstyle{definition}

\newtheorem*{claim*}{Claim}

\theoremstyle{remark}

\numberwithin{equation}{section}

\newcommand{\Sc}{\mathrm{Sc}}

\newcommand{\id}{\mathrm{I}}

\newcommand{\R}{\mathbb{R}}

\newcommand{\supp}{\operatorname{supp}}

\newcommand{\Z}{\mathbb{Z}}

\newcommand{\sph}{\mathbb{S}}
\newcommand{\disk}{\mathbb{D}}
\newcommand{\normal}{\mathbf{n}}
\newcommand{\sys}{\operatorname{sys}}
\newcommand{\nbundle}{\mathcal N}
\newcommand{\sff}{\mathrm{II}}
\newcommand{\vol}{\operatorname{vol}}

\newcommand{\interior}[1]{%
	{\kern0pt#1}^{\mathrm{\,o}}%
}

\makeatletter
\let\save@mathaccent\mathaccent
\newcommand*\if@single[3]{%
	\setbox0\hbox{${\mathaccent"0362{#1}}^H$}%
	\setbox2\hbox{${\mathaccent"0362{\kern0pt#1}}^H$}%
	\ifdim\ht0=\ht2 #3\else #2\fi
}
\newcommand*\rel@kern[1]{\kern#1\dimexpr\macc@kerna}
\newcommand*\wideaccent[2]{\@ifnextchar^{{\wide@accent{#1}{#2}{0}}}{\wide@accent{#1}{#2}{1}}}
\newcommand*\wide@accent[3]{\if@single{#2}{\wide@accent@{#1}{#2}{#3}{1}}{\wide@accent@{#1}{#2}{#3}{2}}}
\newcommand*\wide@accent@[4]{%
	\begingroup
	\def\mathaccent##1##2{%
		\let\mathaccent\save@mathaccent
		\if#42 \let\macc@nucleus\first@char \fi
		\setbox\z@\hbox{$\macc@style{\macc@nucleus}_{}$}%
		\setbox\tw@\hbox{$\macc@style{\macc@nucleus}{}_{}$}%
		\dimen@\wd\tw@
		\advance\dimen@-\wd\z@
		\divide\dimen@ 3
		\@tempdima\wd\tw@
		\advance\@tempdima-\scriptspace
		\divide\@tempdima 10
		\advance\dimen@-\@tempdima
		\ifdim\dimen@>\z@ \dimen@0pt\fi
		\rel@kern{0.6}\kern-\dimen@
		\if#41
		#1{\rel@kern{-0.6}\kern\dimen@\macc@nucleus\rel@kern{0.4}\kern\dimen@}%
		\advance\dimen@0.4\dimexpr\macc@kerna
		\let\final@kern#3%
		\ifdim\dimen@<\z@ \let\final@kern1\fi
		\if\final@kern1 \kern-\dimen@\fi
		\else
		#1{\rel@kern{-0.6}\kern\dimen@#2}%
		\fi
	}%
	\macc@depth\@ne
	\let\math@bgroup\@empty \let\math@egroup\macc@set@skewchar
	\mathsurround\z@ \frozen@everymath{\mathgroup\macc@group\relax}%
	\macc@set@skewchar\relax
	\let\mathaccentV\macc@nested@a
	\if#41
	\macc@nested@a\relax111{#2}%
	\else
	\def\gobble@till@marker##1\endmarker{}%
	\futurelet\first@char\gobble@till@marker#2\endmarker
	\ifcat\noexpand\first@char A\else
	\def\first@char{}%
	\fi
	\macc@nested@a\relax111{\first@char}%
	\fi
	\endgroup
}
\makeatother

\newcommand\overbar{\wideaccent\overline}

\makeatletter
\newcommand*{\transpose}{%
	{\mathpalette\@transpose{}}%
}
\newcommand*{\@transpose}[2]{%
	\raisebox{\depth}{$\m@th#1\intercal$}%
}
\makeatother

\begin{document}
		\title{A $2$-systolic inequality for $\mathbb S^2\times P$}
	
		\author{Jinmin Wang}
		\address[Jinmin Wang]{Institute of Mathematics, Chinese Academy of Sciences}
		\email{jinmin@amss.ac.cn}
\thanks{The first author is partially supported by NSFC 12501169.}
				\author{Zhizhang Xie}
		\address[Zhizhang Xie]{Texas A\&M University}
		\email{xie@tamu.edu}
	
		\begin{abstract} We prove a sharp $2$-systolic inequality for four-dimensional products $\sph^2\times P$, where $P\subset\mathbb R^2$ is an arbitrary convex polygon. Let \[ h=g_{\sph^ 2}+g_{\mathrm{eu}} \] be the standard product metric. If a Riemannian metric $g$ on $\sph^2\times P$ has scalar curvature $\geq \sigma>0$, nonnegative mean curvature on every codimension one face, and dihedral angles no larger than the corresponding dihedral angles of $h$, then both its homotopy and homology $2$-systoles are at most $ \frac{8\pi}{\sigma}.$  This confirms a conjecture of Gromov.
        \end{abstract}
		\maketitle
        
	\section{Introduction}

	Let $(M,g)$ be a compact Riemannian manifold,  possibly with corners. Its $\pi_2$-systole is defined as
	\begin{equation}
		\sys_{2,\pi}(M,g)=\inf\{\textup{Area}(\sph^2,f^*g):f\colon \sph^2\to M\textup{ is  non-zero in } \pi_2(M)\}.
	\end{equation}
	Similarly, its $H_2$-systole is defined as
	\begin{equation}
		\sys_{2,H}(M,g)=\inf\{\textup{Area}_g(\Sigma):[\Sigma]\in H_2(M;\Z)\textup{ is non-zero}\}.
	\end{equation}

    In this paper, we consider four-manifolds of the form \[ M=\sph^2\times\mathcal C, \] where $\mathcal C$ is a compact contractible surface with either smooth or polygonal boundary. In this setting, \[ \pi_2(M)\cong H_2(M;\mathbb Z)\cong\mathbb Z. \] We write $\sys_2(M, g)$ when a statement holds for both $\sys_{2,\pi}(M,g)$ and $\sys_{2,H}(M,g)$. In \cite{GromovGNOSC26}, Gromov conjectured a sharp upper bound for the $2$-systoles of Riemannian metrics on $\sph^2\times P$, where $P$ is a convex polyhedron, in terms of a positive lower bound for the scalar curvature, assuming nonnegative mean curvature of the codimension-one faces and the corresponding dihedral-angle comparison.  Our first main result confirms the conjecture for every convex polygon $P$.

	\begin{theorem}\label{thm:sysPolygon}
		Let $P$ be a convex polygon in $\R^2$. Let $M^4=\sph^2\times P$  and $h=g_{\sph^2}+g_{eu}$ the product metric, where $g_{\sph^2}$ is the standard round metric on the unit sphere $\sph^2$ and $g_{eu}$ is the Euclidean metric on $P$. Let $g$ be a Riemannian metric on $M$. Assume that
		\begin{itemize}
			\item (mean curvature) $H_F(g)\geq 0$ for every codimension one face $F$, 
			\item (scalar curvature) $\Sc(g)\geq \sigma>0$, and
			\item (dihedral angles) $\theta_i(g)\leq \theta_i(h)$, where $\theta_i(g)$ denotes the $i$-th dihedral angle of $g$, and $\theta_i(h)$ denotes the corresponding angle of $h$, equal to the $i$-th angle of $P$.
		\end{itemize}
		Then we have 
        \[ \sys_2(M, g)\leq \frac{8\pi}{\sigma}. \]
	\end{theorem}
	
	$2$-systolic inequalities under scalar curvature assumptions have been extensively studied in the literature. For three dimensional manifolds, Bray, Brendle, Eichmair, Neves \cite{BBEN10,BBA10} studied the least area of embedded $\R P^2$ and $\sph^2$ in three manifolds. More recently, Xu \cite{Xu25} established a topological gap phenomenon for $3$-manifolds that are not covered by $\sph^2\times\sph^1$. He--Li \cite{HeLi26} proved a relative $\pi_2$-systolic inequality for three-dimensional prisms.  In higher dimensions, Zhu \cite{Zhu20,Zhu23} obtained sharp $2$-systolic estimates for nontrivial $2$-spheres in manifolds modeled on $\sph^2\times\mathbb T^n$. Using Dirac-operator methods, Orikasa \cite{Orikasa}, Stryker \cite{Stryker}, and Cecchini--Hirsch--Zeidler \cite{CHZ26} studied stable systolic inequalities, obtaining a sharp stable $2$-systolic inequality for $\mathbb CP^n$ as well as non-sharp estimates for several other classes of manifolds.

	
	Our proof uses capillary minimal hypersurfaces. The approach we take is largely motivated by the work of Wang--Wang--Zhu \cite{WangWangZhu} on the scalar-mean rigidity for Euclidean balls. See also the work of  Ko--Yao \cite{KoYao}. One of the main difficulties of adapting the approach of Wang--Wang--Zhu \cite{WangWangZhu} to the geometric setup of the current paper is due to the existence of singularities from the corner structures of the metrics.  A similar challenge also appears in Gromov's dihedral rigidity conjecture, which was proved by Wang--Xie--Yu \cite{Wang:2021tq,Wang:2022vf,WangXieYu26} by establishing an index theory of Dirac type operator on polyhedral manifolds. However,  the Dirac operator approach does not yet seem to apply to $2$-systolic inequalities on manifolds such as $\sph^2\times P$. On the other hand, the capillary minimal hypersurface theory does not seem to be fully developed on spaces with singular boundaries. In this paper, instead of working directly on $\sph^2\times P$, we shall first make careful   approximations of $\sph^2\times P$ by manifolds with smooth boundary. This should be compared to the smooth approximation approach proposed by Gromov for his dihedral rigidity conjecture, which was later further developed by Brendle \cite{Brendle} and Bi \cite{Bi26} to reduce the problem to the case of  manifolds with smooth boundary and mean curvature comparison in the integral sense. In the current paper,  we give another smoothing procedure. In particular,  we combine geometric smooth approximations with conformal changes of metrics to reduce Gromov's $2$-systole conjecture to the case of manifolds with smooth boundary and \emph{pointwise} mean curvature comparison.  The latter is essential for the existence of capillary minimal hypersurfaces, which is the initial step of the dimension reduction argument in the sense of Schoen-Yau. 
    

    The same method also proves the following smooth-boundary version of Gromov's $2$-systole conjecture, where the  model spaces are  direct products of $\sph^2$ and Euclidean convex smooth domains.
	
	\begin{theorem}\label{thm:sysSmooth}
		Let $D$ be a compact convex domain in $\R^2$. Let $M^4=\sph^2\times D$ and $h=g_{\sph^2}+g_{eu}$. Let $g$ be a smooth Riemannian metric on $M$, and let $f\colon (\partial M,g)\to (\partial D,h)$ be the projection map. Assume that
		\begin{itemize}
			\item $H(g)\geq |df|\cdot H(h)$, and
			\item $\Sc(g)\geq \sigma>0$.
		\end{itemize}
		Then $sys_2(g)\leq \frac{8\pi}{\sigma}$.
	\end{theorem}

	This paper is organized as follows. In Section \ref{sec:reduce}, we reduce Theorem \ref{thm:sysPolygon} and Theorem \ref{thm:sysSmooth} to a weaker $2$-systole inequality (Theorem \ref{thm:sysStrict}), where the comparisons of scalar curvature and mean curvature are assumed to be strict. Then we prove Theorem \ref{thm:sysStrict}  in Section \ref{sec:capillary} via a dimension reduction argument in the sense of Schoen-Yau, by first finding a capillary minimal hypersurface, and then a minimal surface inside that capillary minimal hypersurface to realize the systolic inequality.
	
	\section{Reduction of Theorems \ref{thm:sysPolygon} and \ref{thm:sysSmooth}  to a strict comparison setting}\label{sec:reduce}

    We begin with the following weaker form of the $2$-systolic inequality, in
which the scalar-curvature and mean-curvature inequalities are assumed to be strict. Its proof will be given in
Section~\ref{sec:capillary}.

\begin{theorem}\label{thm:sysStrict}
	Let $M^4=\sph^2\times\disk$, and let $g$ be a smooth Riemannian metric
	on $M$. Let $\sph^1$ denote the standard unit circle, and let
	$f:\partial M\to\sph^1$ be a smooth map such that
	$f^*([d\theta])$ represents a generator of
	$H^1(\partial M;\mathbb Z)$ in the $\sph^1$-direction. If
	there exist constants $\sigma>0$ and $\delta>0$ such that
	\[
		H(g)\geq |df|+\delta,
		\qquad
		\Sc(g)\geq \sigma+\delta>0,
	\]
    then 
	\[
		\sys_2(g)
		\leq \frac{8\pi}{\sigma}.
	\]
\end{theorem}

In this section, we deduce Theorems~\ref{thm:sysSmooth}
and~\ref{thm:sysPolygon} from Theorem~\ref{thm:sysStrict}.

\subsection{Reduction from a convex domain to the standard disk}

In this subsection, we show that Theorem~\ref{thm:sysStrict} implies
Theorem~\ref{thm:sysSmooth}. Let $(M,g) = (\mathbb S^2\times D, g)$ be as in Theorem \ref{thm:sysSmooth}, and let \(\nu\) be the outward unit normal along \(\partial M\). Denote the interior of $M$ by $\mathring M$. 

Since $\partial M$ is smooth, there exists a smooth positive function $\phi$ on $M$ such that \[
	\partial_\nu\phi\geq c>0\quad\text{on }\partial M
	\] 
    for some constant $c>0$.
 For $\varepsilon>0$, define
\[
	g_\varepsilon=e^{2\varepsilon\phi}g.
\]

In dimension four, scalar curvature and mean curvature transform
under this conformal change according to the formulas:
\begin{align}
	\Sc(g_\varepsilon)
	&=e^{-2\varepsilon\phi}
	\left(
		\Sc(g)-6\varepsilon\Delta_g\phi
		-6\varepsilon^2|\nabla^g\phi|^2
	\right),
	\label{eq:conformal-scalar}\\
	H(g_\varepsilon)
	&=e^{-\varepsilon\phi}
	\left(
		H(g)+3\varepsilon\partial_\nu\phi
	\right).
	\label{eq:conformal-mean}
\end{align}

It follows from the above formulas that  there exist constants $a_1>0$ and $a_2>0$ such that, for all sufficiently
small $\varepsilon>0$,
\begin{equation}\label{eq:smooth-perturbation}
	\Sc(g_\varepsilon)\geq \sigma-a_1\varepsilon
	\textup{ and }
	H(g_\varepsilon)
	\geq |df|_{g_\varepsilon}\cdot H(h)+a_2\varepsilon.
\end{equation}

Recall that $D\subset\mathbb R^2$ is a smooth convex domain equipped with
the Euclidean metric. Let
\[
	G:(\partial D,g_{eu}|_{\partial D})
	\longrightarrow(\sph^1,g_{\sph^1}),
	\qquad
	G(p)=\widetilde \nu(p),
\]
be the Gauss map of the convex plane curve $\partial D$, where $\widetilde \nu$
denotes its outward unit normal. The differential of $G$
satisfies
\[
	|dG|=H(h),
\]
and $G$ has degree one. It follows that
\[
	\widetilde f:=G\circ f:\partial M\longrightarrow\sph^1
\]
also represents a generator of $H^1(\partial M;\mathbb Z)$ in the
$\sph^1$-direction. Moreover,
\[
	|d\widetilde f|_{g_\varepsilon}
	\leq |df|_{g_\varepsilon}\cdot H(h).
\]
It follows from \eqref{eq:smooth-perturbation} that 
\[
	H(g_\varepsilon)
	\geq |d\widetilde f|_{g_\varepsilon}+a_2\varepsilon.
\]

Let us write 
\[
	\delta=a_2\varepsilon
	\textup{ and }
	\widetilde{\sigma}
	=\sigma-(a_1+a_2)\varepsilon.
\]
For all sufficiently small $\varepsilon>0$, we have
$\widetilde{\sigma}>0$, and
\[
	\Sc(g_\varepsilon)
	\geq \widetilde{\sigma}+\delta,
	\qquad
	H(g_\varepsilon)
	\geq |d\widetilde f|_{g_\varepsilon}+\delta.
\]
Applying Theorem~\ref{thm:sysStrict} to $(M,g_\varepsilon)$ and
$\widetilde f$, we obtain
\[
	\operatorname{sys}_2(g_\varepsilon)
	\leq
	\frac{8\pi}{\widetilde{\sigma}}
	=
	\frac{8\pi}{\sigma-(a_1+a_2)\varepsilon}.
\]

Finally, since $g_\varepsilon \to g$ uniformly, we have $\sys_2(g_\varepsilon) \to \sys_2(g)$ as $\varepsilon \to 0$.
Letting $\varepsilon\to0$ therefore gives
\[
	\operatorname{sys}_2(g)\leq\frac{8\pi}{\sigma},
\]
which proves Theorem~\ref{thm:sysSmooth}.

	\subsection{Reduction of  Theorem \ref{thm:sysPolygon} to Theorem \ref{thm:sysStrict}}
	
	In this subsection, we show that Theorem \ref{thm:sysStrict} implies Theorem \ref{thm:sysPolygon}.

   \subsubsection{First step: conformal change} Let $(M,g)$ be as in Theorem~\ref{thm:sysPolygon}. Observe that there exists a smooth vector field $\normal$ on a neighborhood
$U$ of $\partial M$ that points strictly outward across every codimension one face of $M$. This for example can be seen as follows.    On the convex polygon $(P, g_{eu})$ with Euclidean metric, choose a point $p$ in the interior of $P$ and consider the vector field 
      \[ v(x) = \frac{x-p}{\|x-p\|}\]
where $x$ lies in a small tubular neighborhood of $\partial P$. This naturally induces a corresponding vector field on $\sph^2\times P$ by setting $\normal(y, x) = v(x)$ for $(y, x)$ in a small tubular neighborhood of $\partial M= \sph^2\times \partial P$. Since by assumption $\theta_i(g) \leq \theta_i(h)$, the normal cone of every codimension two face in $M$ is fiberwise convex. Therefore, there is a diffeomorphism from $(M, h)$ to $(M, g)$ that maps the interior to interior and each face to its corresponding face. In particular, the image of the vector field $\normal$ under this diffeomorphism is a smooth vector field defined on a tubular neighborhood of $\partial M$ in $(M, g) $ that points strictly outward across every codimension one face of $(M, g)$.

Consider a smooth increasing positive function along the flow lines of $\normal$ in a small tubular neighborhood of $\partial M$ and then extend it to the entire $M$. We obtain a smooth function $\psi$ on $M$ and a constant $c>0$
such that 
\[
	\partial_{\nu}\psi\geq c
	\textup{ on every codimension one face $F$ of $M$}, 
\]
where $\nu$ denotes the outward unit normal to $F$. 

For $\varepsilon>0$, set
\[
	g_{\varepsilon}=e^{2\varepsilon\psi}g.
\]
Since $H_F(g)\geq 0$, it follows from \eqref{eq:conformal-scalar} and \eqref{eq:conformal-mean} that  there exist constants $a_1>0$ and $a_2>0$ such that, for all sufficiently small $\varepsilon>0$,
\begin{equation}
	\Sc(g_\varepsilon)\geq \sigma-a_1\varepsilon
	\textup{ and }
	H_F(g_\varepsilon)
	\geq a_2\varepsilon.
\end{equation}
Note that conformal changes preserve dihedral angles.  Since $\psi>0$, we have
$g_{\varepsilon}\geq g$ and hence 
\[
	\sys_2(g_{\varepsilon})
	\geq \sys_2(g).
\]
To summarize, we have 
\begin{itemize}
	\item $H_F(g_{\varepsilon})\geq a_2\varepsilon$ on every face $F$;
	\item $\Sc(g_{\varepsilon})\geq\sigma-a_1 \varepsilon>0$;
	\item $\theta_i(g_{\varepsilon})=\theta_i(g)\leq\theta_i(h)$ for every $i$; 
	\item and $\sys_2(g_{\varepsilon})\geq\sys_2(g)$.
\end{itemize}

\subsubsection{Second step: smoothing out codimension two corners} Our next step is to  smooth out the codimension-two strata while maintaining quantitative
control of the mean curvature. Denote the codimension-one faces of $(M,g_\varepsilon)$ by
$F_1,\ldots,F_N$ so that 
$F_{N+1}=F_1$. We denote 
\[
	E_i=F_i\cap F_{i+1}.
\]
Topologically, we have 
\[
	F_i\cong\sph^2\times[0,1]
	\textup{ and }
	E_i\cong\sph^2.
\]
The two normal vector fields to $E_i$ within $F_i$ and $F_{i+1}$ give a bundle framing of the normal bundle $\nbundle E_i$ of $E_i$ in $M$. So we have the bundle identification
\[
	\nbundle E_i\cong E_i\times\mathbb R^2
\]
and we denote the bundle projection by $\pi_i\colon \nbundle E_i\to E_i$.

Note that we have a splitting
\[
	T_{(x,v)}\nbundle E_i\cong T_xE_i\oplus \nbundle_xE_i.
\]
We equip $\nbundle E_i$ with the Riemannian metric $g_{\nbundle}$ defined by
\[
	(g_{\nbundle})_{(x,v)}
	\bigl((u_1,w_1),(u_2,w_2)\bigr)
	=
	g_\varepsilon (u_1,u_2)+g_\varepsilon (w_1,w_2).
\]
In particular, $g_{\nbundle}$ coincides with $g_\varepsilon$ along the zero section of $\nbundle E_i$ under the
natural identification
\begin{equation}\label{eq:bundleidentify}
    T_{(x, 0)}\nbundle E_i\cong T_xE_i\oplus \nbundle_xE_i=T_xM.
\end{equation}

Let
\[
	\theta_i(x)=\theta_i(g_\varepsilon)(x) \textup{ and }
	\beta_i(x)=\pi-\theta_i(x)
\]
be the dihedral-angle and turning-angle functions along $E_i$,
respectively. We also write 
\[
	\beta_i^0=\pi-\theta_i(h).
\]
Since $\theta_i(g_\varepsilon)\leq\theta_i(h)$, we have
\[
	\beta_i(x)\geq\beta_i^0.
\]
By compactness, there exist constants
$\beta_-,\beta_+\in(0,\pi)$ such that
\[
	\beta_i(x)\in[\beta_-,\beta_+]
\]
for every $i$ and every $x\in E_i$.

Using fiberwise polar coordinates $(r,\vartheta)$ on $\nbundle E_i$, consider
the region
\[
	V_i
	=
	\left\{
		(x,r,\vartheta) \in \nbundle E_i:
		x\in E_i,\ 
		0\leq r\leq r_0,\ 
		0\leq\vartheta\leq\beta_i(x)
	\right\},
\]
where $r_0>0$ is a fixed sufficiently small constant. Denote its codimension one 
faces by
\[
	G_0=\{\vartheta=0\}
	\textup{ and }
	G_1=\{\vartheta=\beta_i(x)\}.
\]
On these faces, define
\[
	T_i=-\partial_r
	\quad\text{on }G_0,
	\qquad
	T_{i+1}=\partial_r
	\quad\text{on }G_1.
\]

For each $i$, choose a smooth diffeomorphism
\[
	\Psi_i:F_i\longrightarrow\sph^2\times[0,1]
\]
such that
\[
	\Psi_i(E_{i-1})=\sph^2\times\{0\}
	\textup{ and }
	\Psi_i(E_i)=\sph^2\times\{1\}.
\]
We denote the parameter for $[0, 1]$ by $t$ and define
\[
	Y_i=(d\Psi_i)^{-1}(\partial_t).
\]
Moreover, we may assume without loss of generality that  $Y_i$ has unit length with respect to
$g_\varepsilon$ in small tubular neighborhoods
of $E_{i-1}$ and $E_i$ in $F_i$,  and is orthogonal to $E_{i-1}$ and $E_i$.

By shrinking $r_0$ if necessary, there exists  a smooth diffeomorphism
\[
	\Phi_i:V_i\longrightarrow U_i,
\]
where $U_i$ is a neighborhood of $E_i$ in $M$, such that
\begin{itemize}
    \item $\Phi_i(G_0)=U_i\cap F_i$ and 
	$\Phi_i(G_1)=U_i\cap F_{i+1},$
    \item $d\Phi_i(T_i)=Y_i$ on $G_0$ and $d\Phi_i(T_{i+1})=Y_{i+1}$  on $
	G_1.$ 
\end{itemize}
Moreover, we have 
\[
	\Phi_i|_{E_i}=\id,
	\textup{ and }
	d\Phi_i|_{E_i}=\id,
\]
where the latter identity should be understood via the identification \eqref{eq:bundleidentify}.

	 In order to smooth out codimension two corners of $M$, we first construct, in each normal bundle $\nbundle E_i$, a smooth hypersurface interpolating between $G_0$ and $G_1$. Then its image under the map $\Phi_i$ will give a smooth transition from  $F_i$ to $F_{i+1}$. The constants and estimates below may depend on the maps $\Phi_i$ and $\Psi_i$ and on the geometry of $(M,g_\varepsilon)$, but they are independent of the parameter $\rho$ introduced below. 
	
	Choose a nonnegative function 
	$
	\chi\in C_c^\infty(\mathbb R)$ such that 
	$$
	\operatorname{supp}\chi=[-1,1],~
	\chi>0\textup{ on }(-1,1),\textup{ and }
	\int_{\mathbb R}\chi(s)\,ds=1,$$
	 and define $$A(s):=\int_{-1}^s\chi(t)\,dt \textup{ for } s\in \mathbb R.$$ 
     Note that  $A(s)=0$ for $s\leq-1$ and $A(s)=1$ for $s\geq1$. For each $\beta>0$, consider the vector-valued function $T_\beta\colon \R \to \mathbb C =  \R^2$ given by 
     \[ T_\beta(s)=e^{i(\pi - \beta A(s))}, \] 
     and the following curve in $\R^2$:
     \[ \tilde\gamma_\beta(s)=
         \int_{-1}^s T_\beta(t) dt. \]
     By construction, $\tilde\gamma_\beta$ is a straight ray for $s\leq -1$ and $s\geq 1$ respectively. By extending these two straight rays to straight lines, they intersect at angle $\beta$. Translating the curve $\tilde\gamma_\beta$ so that the intersection point of these two straight lines is the origin, and denote the resulting curve  by 
     $\gamma_\beta$. See Figure \ref{fig:gamma-beta}. The curve $\gamma_\beta$ is parametrized by arc length and has curvature $\beta\cdot \chi(s)$ for $s\in [-1, 1]$. Moreover,  $\gamma_\beta$ depends smoothly on both $s$ and $\beta\in[\beta_-,\beta_+]$.
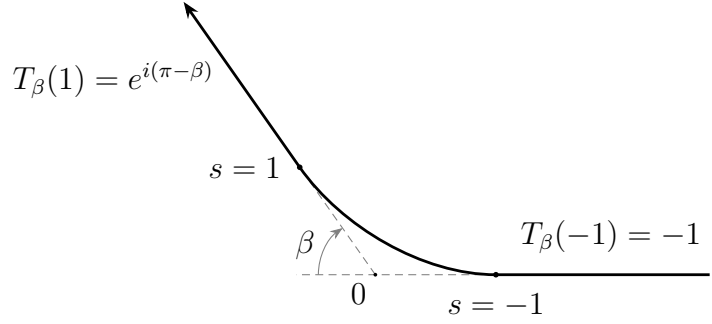
\begin{figure}[ht]
  \centering
  \begin{tikzpicture}[x=1.45cm,y=1.45cm,>=Stealth]
    \coordinate (O)  at (0,0);
    \coordinate (P-) at (1.10,0);
    \coordinate (P+) at (-0.69,0.98);

    \draw[densely dashed,gray]
      (P-) -- (-0.72,0)
      (O) -- (P+);
    \draw[gray,->]
      (-0.52,0) arc[start angle=180,end angle=125,radius=0.52];
    \node at (-0.64,0.27) {$\beta$};
    \fill (O) circle (0.7pt);
    \node[below left=-1pt] at (O) {$0$};

    \draw[line width=1.05pt,-{Stealth[length=2.2mm]}]
      (3.05,0) -- (P-)
      .. controls (0.43,0) and (-0.31,0.44) .. (P+)
      -- (-1.75,2.49);

    \fill (P-) circle (1.05pt);
    \fill (P+) circle (1.05pt);
    \node[below=3pt] at (P-) {$s=-1$};
    \node[left=3pt] at (P+) {$s=1$};
    \node[above=4pt] at (2.15,0)
      {$T_\beta(-1)=-1$};
    \node[left=5pt] at (-1.25,1.78)
      {$T_\beta(1)=e^{i(\pi-\beta)}$};
  \end{tikzpicture}
  \caption{The curve $\gamma_\beta$}
  \label{fig:gamma-beta}
\end{figure}
	
	With the chosen trivialization of $\nbundle E_i$ above, we regard $\gamma_{\beta_i(x)}$ as a curve in the fiber $\nbundle_xE_i$, with its two straight ends tangent to $G_0$ and $G_1$. 
    Define 
    \[ f_i \colon E_i\times[-1,1]\longrightarrow\nbundle E_i \textup{ by } f_i(x,s) = \bigl(x,\gamma_{\beta_i(x)}(s)\bigr), \] and let 
    \begin{equation}\label{eq:smooth-hypersurface}
        \overline\Sigma_i = f_i\bigl(E_i\times[-1,1]\bigr).
    \end{equation}
     Then $\overline\Sigma_i$ is a smooth hypersurface in $\nbundle E_i$ that meets $G_0$ and $G_1$ tangentially along $s=-1$ and $s=1$, respectively. Also, $\overbar\Sigma_{i}$ is contained in the region $\{(x,v)\in \nbundle E_i:|v|\leq L\}$ for some $L>0$. 
	
	Define $\overbar\Sigma_{i,\rho}\coloneqq \rho\cdot\overbar \Sigma_i$ for $\rho>0$, i.e. each curve $\gamma_{\beta_i(x)}\subset \nbundle_xE_i$ is rescaled by a factor $\rho$. Then $\overbar\Sigma_{i,\rho}$ is contained in  $\{(x,v)\in \nbundle E_i:|v|\leq L\rho\}$. Moreover,  the fiberwise  curvature of $\overbar\Sigma_{i,\rho}$ is $\rho^{-1}\beta_i(x)\chi(s)$. More precisely, let $\sff_{g_{\nbundle}}(X,Y)=g_{\nbundle}(\nabla^{g_{\nbundle}}_X\nu_{i,\rho} ,Y)$ be the second fundamental form of $\overbar\Sigma_{i,\rho}$, where $\nu_{i,\rho}$ is the outward unit normal vector of $\overbar\Sigma_{i,\rho}$. Then 
    \[ \sff_{g_{\nbundle}}(\partial_{s,\rho},\partial_{s,\rho})=\rho^{-1}\beta_i(x)\chi(s),\]
    where $\partial_{s,\rho}$ represents the unit-length direction parallel to $\partial_s$.
    
    As $\rho\to 0$, we have $\nu_{i,\rho}=n_{i,\rho}+O(\rho)$, where $n_{i,\rho}$ is the outward unit normal vector to the curve $\gamma_{\beta_i(x)}$ in the fiber $\nbundle_x E_i$. Consequently, \[ \sff_{g_{\nbundle}} (\partial_{s,\rho},\partial_{s,\rho}) = \rho^{-1}\beta_i(x)\chi(s)+O(1), \] while all remaining components of the second fundamental form, with respect to an orthonormal frame, are uniformly bounded. Therefore, whenever $L\rho<r_0$, there is a uniform constant $C>0$, independent of $\rho$,  such that the mean curvature of $\overbar\Sigma_{i,\rho}$ satisfies
	\begin{equation}
		|H_{\overbar\Sigma_{i,\rho}}(g_{\nbundle})(x, s)-\rho^{-1}\beta_i(x)\chi(s)|\leq C.
	\end{equation}
	Here and below, in order to avoid introducing too many constants, we let $C$ denote a positive constant that may represent different values for different estimates and  may depend on the geometry of $M$ and on the maps $\Phi_i$ and $\Psi_i$, but is always independent of $\rho$.
    
We now define $\Sigma_{i,\rho}=\Phi_i(\overbar\Sigma_{i,\rho})$. When $L\rho<r_0$, the hypersurface $\Sigma_{i,\rho}$ is contained in the $L'\rho$-neighborhood of $E_i$ for some $L'>0$. Equivalently, we consider $\overbar\Sigma_{i,\rho}$ inside $V_i\subset \nbundle E_i$ with respect to the metric $\Phi_i^*g_\varepsilon$. Denote by $\nu_{i,\rho}^{g_\varepsilon}$ its outward unit normal vector with respect to $\Phi_i^*g_\varepsilon$.  Since $\Phi_i^*g_\varepsilon$ agrees with
$g_{\nbundle}$ along the zero section $E_i$ and
$\overbar\Sigma_{i,\rho}$ is at distance $O(\rho)$ from $E_i$,
we have
\[
	\nu_{i,\rho}^{\,g_\varepsilon}
	=
	n_{i,\rho}+O(\rho).
\] 
Therefore, whenever $L\rho<r_0$, we have
	\begin{equation}\label{eq:mean-estimate}
		|H_{\overbar\Sigma_{i,\rho}}(\Phi_i^*g_\varepsilon)(x, s)-\rho^{-1}\beta_i(x)\chi(s)|\leq C.
	\end{equation}

By construction, $\Sigma_{i,\rho}$ gives a smooth transition from  $F_i$ to $F_{i+1}$. Define $\Sigma_\rho$ to be the union of all the hypersurfaces $\Sigma_{i,\rho}$ and parts of the faces $F_i$ that are \emph{not} outside of $\Sigma_{i,\rho}$;  this is a closed smooth three-dimensional manifold. Let $M_\rho$ be submanifold  with smooth boundary $\Sigma_\rho$ contained in $M$.

Let us now define a map
\begin{equation}\label{eq:f_rho}
    f_\rho\colon \Sigma_\rho\longrightarrow\sph^1
\end{equation}
as follows. Recall that 
\[
	\beta_i^0=\pi-\theta_i(h)
\]
is the turning angles of the convex Euclidean $(P, g_{eu})$. We have 
\[
	\sum_{i=1}^N\beta_i^0=2\pi.
\]
We denote 
\[
	\alpha_0=0
\textup{ and }	\alpha_i=\sum_{j=1}^i\beta_j^0 
\]
so that $\alpha_N=2\pi$. On $\Sigma_{i,\rho}$,  we define
\begin{equation}\label{eq:def-boundary-map}
	f_\rho\left(
		\Phi_i\bigl(x,\rho\gamma_{\beta_i(x)}(s)\bigr)
	\right)
	\coloneqq 
	\exp\left(
		i\bigl(\alpha_{i-1}+\beta_i^0A(s)\bigr)
	\right).
\end{equation}
On  $F_i\cap \Sigma_\rho$, we define
\[
	f_\rho\equiv e^{i\alpha_{i-1}}.
\]
Since $A$ is constant for $s\leq-1$ and $s\geq1$, the piecewise definitions above  fit
together  to give a smooth map. By construction,  $f_\rho$ has degree one in the $\sph^1$-direction.

On $\Sigma_{i,\rho}$, we have
	\begin{equation}
		df_\rho=i\beta^0_i\chi(s) f_\rho\cdot (\Phi_i^{-1})^\ast( ds),
	\end{equation}
	which is supported in the $L'\rho$-neighborhood of $E_i$. By construction, we have $(d\Phi_i^{-1})^\ast= I +O(\rho)$, where $I$ is the identity map,  and 
    \[ |ds|=\frac{1}{\rho}+O(1) \] as $\rho\to 0$. Hence there is a constant $C>0$ independent of $\rho$ such that 
	\begin{equation*}
		|df_\rho|\leq \frac{\beta_i^0}{\rho}\chi(s)+C
	\end{equation*}
	on $\Sigma_{i,\rho}$. Together with \eqref{eq:mean-estimate} and the fact that $\beta_i(x)\geq \beta^0_i$, this implies that
	\begin{equation*}
		H_{\Sigma_{i,\rho}}(g_\varepsilon)-|df_\rho|\geq a_2\varepsilon-C.
	\end{equation*}
On $F_i\cap \Sigma_\rho$, $f_\rho$ is constant. Hence  
\[ H_{\Sigma_{\rho}}(g_\varepsilon)-|df_\rho| = H_{F_i}(g_\varepsilon) \geq a_2\varepsilon \textup{ on } F_i\cap \Sigma_\rho.\]

Observe that there exists $C>0$ independent of $\rho$ such that each $\Sigma_{i, \rho}$ has volume 
\[
	\vol_{g_\varepsilon}(\Sigma_{i, \rho})\leq C\rho.
\]

It follows from the above discussion that there exist a nonnegative
smooth function $k_\rho$ on $\Sigma_\rho$ and constants $c_1,c_2>0$,
independent of $\rho$, such that
\begin{equation}\label{eq:k_rho}
\begin{split}
H_{\Sigma_\rho}(g_\varepsilon)-|df_\rho|
	&\geq a_2\varepsilon- 3k_\rho, \\
0\leq k_\rho&\leq c_1, \\
	\vol_{g_\varepsilon}(\supp k_\rho)
	&\leq c_2\rho,
    \end{split}
\end{equation}
where $\supp k_\rho$ is the support of $k_\rho$.
For example, one may take $k_\rho$ to be a smooth cutoff function equal to sufficiently large positive constant on a small $O(\rho)$-neighborhood of each 
$\Sigma_{i, \rho}$ in $\Sigma_\rho$. In particular, for every $p>0$,
\begin{equation}\label{eq:h-Lp-bound}
	\|k_\rho\|_{L^p(\Sigma_\rho)}
	\leq
	c_1(c_2\rho)^{1/p}
	\to 0
	\textup{ as }\rho\to0.
\end{equation}

Recall that $g_\varepsilon =e^{2\varepsilon\psi}g$ with
$\psi>0$. Therefore, we have 
\[
	\operatorname{sys}_2(M_\rho,g_\varepsilon)
	\geq
	\operatorname{sys}_2(M,g_\varepsilon)
	\geq
	\operatorname{sys}_2(M,g).
\]
To summarize, for  sufficiently small $\varepsilon>0$ and $\rho>0$, we have constructed a compact Riemannian manifold with smooth
boundary $(M_\rho,\Sigma_\rho,g_\varepsilon)$ and smooth maps
$f_\rho:\Sigma_\rho\to\sph^1$ such that 
\begin{enumerate}
	\item $\Sc(g_\varepsilon)\geq\sigma-a_1\varepsilon;$
	\item
		$H_{\Sigma_\rho}(g_\varepsilon)-|df_\rho|
		\geq a_2\varepsilon-3k_\rho,$
	where $k_\rho$ is a nonnegative smooth function on $\Sigma_\rho$ satisfying 
	\[
		0\leq k_\rho\leq c_1 \textup{ and }
		\vol_{g_\varepsilon}(\operatorname{supp}k_\rho)
		\leq c_2\rho;
	\]
	\item  and $ \sys_2(M_\rho,g_\varepsilon)
		\geq\sys_2(M,g)$
\end{enumerate}
where $a_1, a_2$ are positive constants independent of $\varepsilon$ and $\rho$, and , $c_1,c_2$ are independent of $\rho$.

Finally, by the construction of $\Sigma_\rho$, there is a bi-Lipschitz homeomorphism
\[ L_\rho\colon (M_\rho,g_\varepsilon)\to (M,g_\varepsilon) \]
whose bi-Lipshitz constants are uniformly bounded in $\rho$. For example, in the normal bundle $\nbundle E_i$,  fix a local bi-Lipschitz map that maps a small tubular neighborhood of $\overbar \Sigma_{i}$ from \eqref{eq:smooth-hypersurface} to a small tubular neighborhood of $(G_0\cup G_1)$ (see Figure \ref{fig:rounded-sector-map}). Fiberwise rescaling of this map keeps the bi-Lipschitz constants uniformly bounded. Using a partition of unity to glue these local bi-Lipschitz maps with the identity map away from these neighborhoods of $E_i$ gives the desired bi-Lipschitz homeomorphism $ L_\rho\colon (M_\rho,g_\varepsilon)\to (M,g_\varepsilon)$ above. 

Since $L_\rho\colon (M_\rho,g_\varepsilon)\to (M,g_\varepsilon)$ is bi-Lipschitz, the Sobolev $H^1$-spaces for the manifolds $M_\rho$ and $M$ can  be identified, and their Sobolev $H^1$-norms are equivalent.


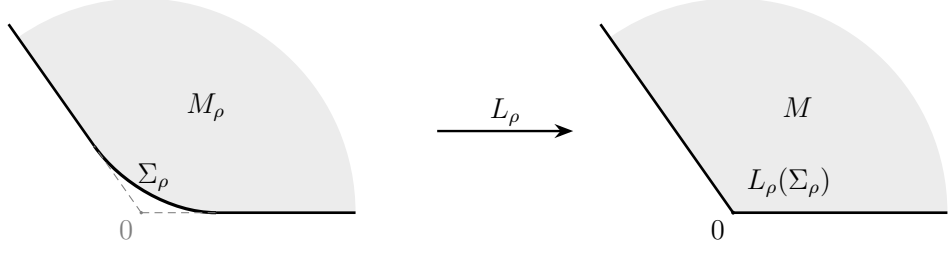
\begin{figure}[ht]
\centering
\begin{tikzpicture}[scale=0.9,transform shape,>=Stealth]

  \begin{scope}
    \coordinate (O)  at (0,0);
    \coordinate (P-) at (1.10,0);
    \coordinate (P+) at (-0.69,0.98);
    \coordinate (Q-) at (3.15,0);
    \coordinate (Q+) at (-1.95,2.77);

    \fill[gray!15,draw=none]
      (P-)
      -- (Q-)
      arc[start angle=0,end angle=125,radius=3.15]
      -- (Q+)
      -- (P+)
      .. controls (-0.31,0.44) and (0.43,0) .. (P-)
      -- cycle;

    \draw[line width=1pt] (P-) -- (Q-);
    \draw[line width=1pt] (P+) -- (Q+);

    \draw[line width=1.2pt]
      (P+)
      .. controls (-0.31,0.44) and (0.43,0) .. (P-);

    \draw[densely dashed,gray] (O) -- (P-);
    \draw[densely dashed,gray] (O) -- (P+);

    \fill[gray] (O) circle (0.7pt);
    \node[below left=-1pt,gray] at (O) {$0$};

    \node at (0.95,1.55) {$M_\rho$};
    \node at (0.18,0.52) {$\Sigma_\rho$};

  \end{scope}

  \draw[->,line width=0.9pt] (4.35,1.2) -- (6.35,1.2);
  \node at (5.35,1.48) {$L_\rho$};

  \begin{scope}[shift={(8.7,0)}]

    \coordinate (O2)  at (0,0);
    \coordinate (Q2-) at (3.15,0);
    \coordinate (Q2+) at (-1.95,2.77);
    \coordinate (P2-) at (1.10,0);
    \coordinate (P2+) at (-0.69,0.98);

    \fill[gray!15,draw=none]
      (O2)
      -- (Q2-)
      arc[start angle=0,end angle=125,radius=3.15]
      -- (Q2+)
      -- cycle;

    \draw[line width=1pt] (O2)--(Q2-);
    \draw[line width=1pt] (O2)--(Q2+);


    \fill (O2) circle (0.8pt);
    \node[below left=-1pt] at (O2) {$0$};

    \node at (0.95,1.55) {$M$};
    \node at (0.82,0.45) {$L_\rho(\Sigma_\rho)$};

  \end{scope}

\end{tikzpicture}
\caption{The smooth boundary $\Sigma_\rho$
of $M_\rho$ is mapped by the Lipschitz map $L_\rho$ to the two boundary
faces at the corner.}
\label{fig:rounded-sector-map}
\end{figure}

\subsubsection{Step three: another conformal change to improve the mean curvature}
	
	Recall  that the mean curvature $\Sigma_\rho = \partial M_\rho$ satisfies: 
    \[ H_{\Sigma_\rho}(g_\varepsilon)-|df_\rho|
		\geq a_2\varepsilon-3k_\rho. \]
        Our next step is to perform another conformal change on $M_\rho$ to improve the above inequality so that $H_{\Sigma_\rho}(g_\varepsilon)-|df_\rho|$ becomes positive everywhere on $\Sigma_\rho$.
	
	\begin{lemma}\label{lemma:conformal}
		Let $(M_\rho,\partial M_\rho,g_\varepsilon)$ be  as above. Let $k_\rho$ be the function from \eqref{eq:k_rho}. Then there exists positive $\alpha_\rho=O(\rho^{1/3})$ such that the first eigenvalue of the following problem is non-negative:
		\begin{equation}\label{eq:laplace}
		    \begin{cases}
		-\Delta\psi + \alpha_\rho\psi_\rho=\lambda_\rho\psi_\rho\\
		\partial_{\nu_\rho}\psi_\rho=k_\rho\psi_\rho.
		\end{cases}
		\end{equation}
		Furthermore, as $\rho\to 0$, we have $\frac{\max\psi_\rho}{\min\psi_\rho}\to 1$.
	\end{lemma}
	\begin{proof}	
		Using the uniform Lipschitz maps $L_{\rho_0}^{-1}L_\rho\colon M_\rho\to M_{\rho_0}$, we see that the Sobolev $H^1$-norm on all $M_\rho$'s are uniformly equivalent, hence the Sobolev $H^1$-space for all $M_\rho$'s are equal. Furthermore, based on the standard Sobolev inequalities on $M_{\rho_0}$, we have the following trace inequality for any $\rho\leq \rho_0$ with uniform coefficient:
		\begin{equation}\label{eq:trace2}
			\|\varphi\|_{L^3(\partial M_\rho)}\leq c_0\|\varphi\|_{H^1(M_\rho)},~\forall \varphi\in H^1(M_\rho),
		\end{equation}
		and
		\begin{equation}\label{eq:trace1}
			\|\psi\|_{L^1(\partial M_\rho)}\leq c_1\|\psi\|_{W^{1,1}(M_\rho)},~\forall \psi\in W^{1,1}(M_\rho).
		\end{equation}

		Let $\alpha_\rho=c_0^2\|k_\rho\|_{L^3(\partial M_\rho)}=O(\rho^{1/3})$. Then by the H\"older inequality, we have
			\begin{equation}
	\begin{split}
				\int_{\partial M_{\rho}} |k_\rho|\varphi^2\leq& \| k_\rho\|_{L^{3}(\partial M_{\rho})}\|\varphi^2\|_{L^{3/2}(\partial M_{\rho})}=\|k_\rho\|_{L^{3}(\partial M_{\rho})}\|\varphi\|_{L^{3}(\partial M_{\rho})}^2\\
				\leq& c_0^2\|k_\rho\|_{L^{3}(\partial M_{\rho})}\|\varphi\|_{H^1(M_\rho)}^2
				=\alpha_\rho\left(\int_{M_\rho}|\nabla\varphi|^2+|\varphi|^2\right)
	\end{split}
		\end{equation}

		Let $\psi_\rho$ be an eigenfunction of the first eigenvalue of \eqref{eq:laplace}. Then
		\begin{align*}
			\lambda_\rho\int_{M_\rho}|\psi_\rho|^2&=\int_{M_\rho} (|\nabla\psi_\rho|^2+\alpha_\rho|\psi_\rho|^2)-\int_{\partial M_\rho}k_\rho\psi_\rho^2\\
			&\geq \int_{M_\rho} (1-\alpha_\rho)|\nabla\psi_\rho|^2\geq 0.
		\end{align*}
		Thus $\lambda_\rho\geq 0$.
		
		Without loss of generality, assume that $\|\psi_\rho\|_{L^2(M_\rho)}=1$. Consider the bilinear form
		\begin{equation}
			B_\rho(u,v)=\int_{M_\rho}\langle\nabla u,\nabla v\rangle-\int_{\partial M_\rho}k_\rho uv.
		\end{equation}
		Let $Q_\rho(u)=B_\rho(u,u)$. 
		The function $\psi_\rho$ minimizes $Q_\rho$ and satisfies $Q_\rho(\psi_\rho)=\lambda_\rho-\alpha_\rho$. Since
		$$|\lambda_\rho-\alpha_\rho|\leq |Q_\rho(\frac{1}{\sqrt{vol(M_\rho)}})|\leq \frac{1}{vol(M_\rho)}\int_{\partial M_\rho}k_\rho\leq \alpha_\rho,$$
		it follows that $\lambda_\rho\leq 2\alpha_\rho$. Therefore,
		\begin{equation}
			\int_{M_\rho}|\nabla\psi_\rho|^2\leq \frac{\lambda_\rho}{1-\alpha_\rho}\to 0.
		\end{equation}
		In particular, if we set $\overbar\psi_\rho\coloneqq (L_\rho^{-1})^*\psi_\rho\in H^1(M)$. then
		\begin{equation}
			\int_{M_\rho}|\nabla\overbar \psi_\rho|^2\to 0.
		\end{equation}

		It follows that $\{\overbar\psi_\rho\}$ is uniformly bounded in the Sobolev $H^1$-norm. We may assume that $\{\overbar\psi_\rho\}$ is a Cauchy sequence in the $L^2$-norm. Moreover, since $\nabla\overbar\psi_\rho\to 0$, $\overbar\psi_\rho$ is also a Cauchy sequence in the $H^1$-norm. In particular, $\overbar\psi_\rho$ converges to the constant function $vol(M)^{-1/2}$ in the $H^1$-norm.
		
		By the uniform trace inequality, uniform Sobolev inequality and the Moser iteration, we have that $\overbar\psi_\rho\to vol(M)^{-1/2}$ in the $L^\infty$-norm. Consequently, we have 
        \[ \frac{\max\psi_\rho}{\min\psi_\rho}\to 1. \]

	\end{proof}
	
	Let $\psi_\rho$ be the eigenfunction of the first eigenvalue of the elliptic equation \eqref{eq:laplace}. $\psi_\rho$ is nowhere vanishing, hence we may assume without loss of generality that $\psi_\rho$ is everywhere positive. Let us now consider the following conformal change of metric 
	\[
	g^\rho=\psi_\rho^2g_\varepsilon
	\]
	on $M_\rho$. With respect to the new metric $g^\rho$,  the map $f_\rho\colon (\partial M_\rho, g^\rho)\to (\sph^1,g_{\sph^1})$ satisfies 
\begin{equation} |df_\rho|_{g^\rho}=\psi_\rho^{-1}|df_\rho|_{g_\varepsilon}.
	\end{equation}
The  conformal change formulas, together with \eqref{eq:laplace} and \eqref{eq:k_rho},  imply 
	\begin{equation}
    \begin{split}
        \Sc(g^\rho)
		& =  \psi_\rho^{-2}
		\bigl(\Sc(g_\varepsilon)-6 \psi_\rho^{-1}\Delta \psi_\rho \bigr)\\
        & = \psi_\rho^{-2}
		\bigl(\Sc(g_\varepsilon)+6(\lambda_\rho-\alpha_\rho)\bigr) \geq\psi_\rho^{-2}\bigl(\Sc(g_\varepsilon)-6\alpha_\rho\bigr)
    \end{split}
	\end{equation}
and
	\begin{equation}
    \begin{split}
        H_{\Sigma_\rho}(g^\rho)&=\psi_\rho^{-1}(H_{\Sigma_\rho}(g_\varepsilon)+3\frac{\partial_{\nu_\rho}\psi_\rho}{\psi_\rho}) \\
		& \geq \psi_\rho^{-1}(|df_\rho|_{g_\varepsilon}+a_2\varepsilon)=|df_\rho|_{g^\rho}+\psi_\rho^{-1}a_2\varepsilon.
    \end{split}
	\end{equation}
	
	Finally, for a fixed $\varepsilon$,  we apply Theorem \ref{thm:sysStrict} to $(M_\rho, g^\rho)$. Since  $\frac{\max\psi_\rho}{\min\psi_\rho}\to 1$, $\psi_\rho$ can be chosen to approach $1$ uniformly, as $\rho\to 0$. Letting  $\rho\to0 $ and  then followed by  $\varepsilon \to0$, we see that Theorem \ref{thm:sysStrict} implies Theorem \ref{thm:sysPolygon}.
	
	\section{Capillary minimal hypersurfaces}\label{sec:capillary}
	
	In this section, we prove Theorem \ref{thm:sysStrict}.
	
	\subsection{First step: capillary minimal hypersurface}
	
	Let 
	$f \colon \partial M\to\sph^1$ be the map given in  Theorem \ref{thm:sysStrict}. We first perform a minor modification of $f$.
    
	\begin{lemma}\label{lemma:poles-open}
		Let $\pm$ be the two poles of $\sph^1$.
		After replacing \(f\) by a homotopic smooth map and decreasing the 
		constant $\delta$ in Theorem \ref{thm:sysStrict},  we may assume that the inverse images \(B_-=f^{-1}(-)\) and
		\(B_+=f^{-1}(+)\) of these antipodal points contain open sets.
	\end{lemma}
	\begin{proof}
		This is the one-dimensional version of the perturbation argument in
		\cite[Lemma 3.1]{WangWangZhu}. Choose a smooth degree-one map
		\(q:\sph^1\to\sph^1\) that is constant on small arcs around the two poles and
		satisfies \(\|dq\|\leq1+\varepsilon\). Then \(q\circ f\) is homotopic to \(f\),
		has the same integral cohomology class, and is constant on the inverse images
		of those arcs. Since the inequalities in Theorem \ref{thm:sysStrict} are
		strict and \(S\) is compact, \(\varepsilon\) can be chosen so that they remain
		strict.
	\end{proof}

    From now on, by possibly working with a slightly smaller positive number $\delta$, we may assume without loss of generality that $f \colon \partial M\to\sph^1$ given in Theorem \ref{thm:sysStrict} satisfies the extra properties listed in Lemma \ref{lemma:poles-open}.
    
	Let \(\Psi\) be the distance  function from the point \(+\) on \(\sph^1\), and set
	\[
	\mu_\partial=\cos(\Psi\circ f).
	\]
	Although \(\Psi\) is not smooth at the poles, \(\mu_\partial\) is smooth.
	
	Let 
	\begin{equation}
		\mathcal C=\{\textup{Caccioppoli sets }\Omega\subset M,~\partial(\overline{\partial^*\Omega\cap\mathring M})\subset \partial M\setminus (B_-\cup B_+),~B_-\subset\Omega\}.
	\end{equation}
	Consider the following energy functional:
	\begin{equation}
		\mathcal A(\Omega)=\mathcal H^3_g(\partial^*\Omega\cap \mathring{M})-\int_{\partial^*\Omega\cap \partial M} \mu_\partial d\mathcal H^3_g
	\end{equation}
	for every $\Omega\in\mathcal C$.
	By \cite[Lemma 2.4]{WangWangZhu}, a minimizer $(\Omega,\partial\Omega,g_\Omega=g|_\Omega)$ exists. In particular, by the assumption that $\partial M$ is strictly mean convex (i.e. $H(g)\geq |df|+\delta$), and $\mu_\partial=\pm 1$ on the open sets $B_\pm$, the maximal principle yields that the minimizer is striclty away from $B_\pm$; see \cite[Lemma 2.4 and Appendix A]{WangWangZhu}. Moreover, 
    since \(\dim M=4\), the minimizer $\Omega$ admits no interior or boundary singularities. The minimizer satisfies the following first  and second variational formulas.
	\begin{enumerate}

		\item First variation: 
		Let $Y=\overline{\partial\Omega\cap\mathring M}$. Then $Y$ is a smooth $3$-manifold with boundary $\partial Y\subset \partial M\setminus(B_+\cup B_-)$. In $M$, the hypersurface $Y$ has zero mean curvature. Moreover, $J(z)=\Psi(f(z))$ for every $z\in \partial Y$, where $J(z)$ is the contact angle between $Y$ and $\partial M$ at $z$.
		\item Stability inequality: 
		For every $\varphi\in C^\infty(Y)$, we have
		\begin{equation}
			\begin{split}
				0\leq &\int_Y (|\nabla\varphi|^2-(\textup{Ric}_g(\nu_Y,\nu_Y)+|A_Y|^2)\varphi^2)d\mathcal H^3_g\\
				&+\int_{\partial Y}(H_{\partial Y}(g)-\frac{H(g)}{\sin J}+\frac{1}{\sin J}\langle \vec n,\nabla(\Psi\circ f)\rangle)\varphi^2 d\mathcal H^2_g,
			\end{split}
		\end{equation}
		where $\textup{Ric}_g$ denotes the Ricci curvature of $g$, $\nu_Y$ is the unit normal vector to $Y\subset M$, $A_Y$ is the second fundamental form of $Y \subset M$, $H_{\partial Y}(g)$ denotes the mean curvature of $\partial Y$ in $Y$, and $\vec n$ denotes the upward unit normal vector to $\partial Y$ in $\partial M$.
	\end{enumerate}
	
	\begin{lemma}\label{lem:first-stability}
		For every
		\(\varphi\in C^\infty(Y)\),
		\begin{align}
			0\leq{}&
			\int_Y\left(
			|\nabla\varphi|^2
			+\frac12\bigl(\Sc_Y-\sigma-\delta\bigr)\varphi^2
			\right)d\mathcal H_g^3 \notag\\
			&+\int_{\partial Y}(H_{\partial Y}-\delta)\varphi^2\,
			d\mathcal H_g^2. \label{eq:first-reduced-stability}
		\end{align}
	\end{lemma}
	\begin{proof}
		Because \(Y\) is minimal, the Gauss equation gives
		\[
		-\bigl(\operatorname{Ric}_g(\nu_Y,\nu_Y)+|A_Y|^2\bigr)
		=-\frac12\Sc(g)+\frac12\Sc_Y-\frac12|A_Y|^2
		\leq\frac12(\Sc_Y-\sigma-\delta-|A_Y|^2).
		\]
		Moreover,
		\(\partial Y\) is a compact subset of \(S\setminus(B_-\cup B_+)\), so
		\(0<J<\pi\) on $\partial Y$. Note that 
		\begin{equation}
			\langle\vec n,\nabla(\Psi\circ f)\rangle_z=\langle\nabla\Psi,df(\vec n)\rangle_{f(z)}.
		\end{equation}
        Since $\Psi$ is the distance function from the point $+$ on $\sph^1$,  we have $|\nabla \Psi| = 1$, hence 
        \[  |\langle\vec n,\nabla(\Psi\circ f)\rangle_z|\leq |df|.\]
		Therefore
		\begin{equation}
			H(g)-\langle\vec n,\nabla(\Psi\circ f)\rangle\geq H(g)-|df|\geq \delta>0
		\end{equation}
		This completes the proof.
	\end{proof}
	
	We now examine the topology of $Y$. Let \(I_0\) and \(I_1\)
	be the two components of \(\sph^1\setminus\{\pm\}\), and set
	\[
	\partial_j Y=f^{-1}(I_j)\subset \partial Y,\qquad j=0,1.
	\]
	Each \(\partial_j Y\) is a union of connected components of \(\partial Y\).
	
	\begin{lemma}\label{lem:first-topology}
		After replacing \(Y\) by one of its connected components, both \(\partial_0 Y\) and
		\(\partial_1 Y\) are nonempty and
		\begin{equation}\label{eq:gamma-nonzero}
			0\neq[\partial_0 Y]\in H_2(M;\mathbb Z).
		\end{equation}
	\end{lemma}
	
	\begin{proof}
	Let us denote 
	\[
		W=\Omega\cap\partial M.
	\]
	By construction, $B_-$ is contained in the interior of $W$, while
	$B_+$ is disjoint from $W$. Since $\partial M$ is compact,
	there exist neighborhoods $J_-$ and $J_+$ of $-$ and $+$,
	respectively, such that
	\[
		f^{-1}(J_-)\subset W
		\textup{ and }
		f^{-1}(J_+)\cap\overline W=\emptyset.
	\]
	
	We choose regular values of the map $f$:
	\[
		x_-\in I_0\cap J_-
		\textup{ and }
		x_+\in I_0\cap J_+,
	\]
	and let $\gamma\subset I_0$ be the closed arc from $x_-$ to $x_+$. Since
	$\partial W$ is disjoint from $f^{-1}(J_-\cup J_+)$, we have
	\[
		\partial_0Y=\partial W\cap f^{-1}(\gamma).
	\]
	The space 	$W\cap f^{-1}(\gamma) \subset \partial M$
	is an oriented cobordism  between $\partial_0Y$ and
	the regular fiber $f^{-1}(a)$.  Consequently,
	\[
		[\partial_0Y]= [f^{-1}(a)]
		\qquad\text{in }H_2(\partial M;\mathbb Z),
	\]
	up to the choices of orientations. 
	Because $a$ is a regular value, $[f^{-1}(a)]$ is the 
	Poincar\'e dual of $f^*([d\theta])$, 
	up to the choice of orientation. Since $f^*([d\theta])$ is a generator
	of $H^1(\partial M;\mathbb Z)$, the class $[f^{-1}(a)]$ is a
	 generator of $H_2(\partial M;\mathbb Z)$. Its image under
	\[
		H_2(\partial M;\mathbb Z)\xrightarrow{\ \cong\ } H_2(M;\mathbb Z)
	\]
	is again a generator. Hence $
		0\neq[\partial_0Y]\in H_2(M;\mathbb Z).$

    Given  a connected component $Y_\lambda$ of $Y$, if its boundary $\partial Y_\lambda$ lies in a single $\partial_j Y$, then $[\partial Y_\lambda] = 0$  in $H_2(M)$, because $\partial Y_\lambda$ bounds $Y_\lambda$. Removing all such connected components leaves a union of connected components intersecting both $\partial_0 Y$ and $\partial_1 Y$. Since $[\partial_0 Y]\neq 0$,  there is at least one  connected component $Y_\lambda$ of $Y$ such that $[\partial_0 Y_\lambda] \neq 0$.  we
		replace $Y$ by $Y_\lambda$. This proves the lemma. 
	\end{proof}
	
	\subsection{Second step: weighted minimal hypersurface}
	Let $Y$ the capillary hypersurface we obtained in the previous subsection. Consider the following operator 
    \[ \mathcal L=-\Delta_Y+\frac12\Sc_Y-\frac 1 2(\sigma+\frac\delta 2)\]
    and the corresponding eigenvalue problem on $Y$:
	\begin{equation}\label{eq:first-eigenproblem}
		\begin{cases}
			\displaystyle \mathcal L\psi=\lambda_1 \psi\qquad
			&\textup{ on }Y,\\
			\partial_{\nu_{\partial Y}}\psi=(-H_{\partial Y}+\delta)\psi
			&\textup{ on }\partial Y.
		\end{cases}
	\end{equation}
	Let $u$ be an eigenfunction of the first eigenvalue. Then $u$ is nowhere vanishing, hence may be chosen to be strictly positive. It follows from 
	\eqref{eq:first-reduced-stability} that  
	\begin{equation}\label{eq:first-eigenvalue}
		\lambda_1\geq\frac{\delta}{4}\geq 0.
	\end{equation}
	Equip $Y\times\sph^1$ with
	\[
	g_1=g|_Y+u^2dt^2.
	\]
	The warped-product formulas and \eqref{eq:first-eigenproblem} give
	\begin{align}
		\Sc(g_1)
		&=\Sc_Y-2\frac{\Delta_Y u}{u}
		\geq\sigma+\frac\delta 2, \label{eq:first-warp-scalar}\\
		H_{\partial(Y\times\sph^1)}(g_1)
		&=H_{\partial Y}+\frac{\partial_{\nu_{\partial Y}}u}{u}
		=\delta. \label{eq:first-warp-mean}
	\end{align}

    \begin{lemma}\label{lem:second-minimizer}
	There exists a smooth closed surface $Z\subset\mathring Y$ such that
	\[
		W=Z\times\sph^1\subset(Y\times\sph^1,g_1)
	\]
	is an $\sph^1$-invariant minimal hypersurface separating
	$\partial_0Y\times\sph^1$ from $\partial_1Y\times\sph^1$.
	Every connected component of $W$ is two-sided and stable. Moreover,
	up to an appropriate choice of  the orientation of $Z$,
	\[
		[Z]=[\partial_0Y]
		\qquad\text{in }H_2(M;\mathbb Z).
	\]
\end{lemma}
	
	\begin{proof}
		We minimize the perimeter among \(\sph^1\)-invariant Caccioppoli sets whose trace is
		one on \(\partial_0Y\times\sph^1\) and zero on
		\(\partial_1Y\times\sph^1\). The strict mean convexity in
		\eqref{eq:first-warp-mean} and the maximum principle keep the minimizer away
		from the boundary. Since $Y\times \mathbb S^1$ has dimension $4$, standard regularity implies that the interior boundary of the minimizer is  a smooth closed hypersurface $w$. Invariance under the circle action implies $W=Z\times\sph^1$. In fact, the minimizing region is $\sph^1$ invariant, hence descends to a region
		in $Y$ whose oriented boundary is $Z-\partial_0 Y$. This proves the last assertion of the lemma. 
	\end{proof}
	
	By Lemmas \ref{lem:first-topology} and \ref{lem:second-minimizer}, at least one
	connected component $Z_\lambda$ of $Z$ represents a nonzero class in
	$H_2(M;\mathbb Z)$. We replace $Z$ by $Z_\lambda$ and
	$W$ by $Z_\lambda\times\sph^1$. For notational simplicity, we continue to denote them by $Z$ and $W$, respectively. Since $W$ is a two-sided stable minimal hypersurface in $(Y\times\sph^1,g_1)$, the stability inequality and the Gauss equation give 
	\begin{equation}\label{eq:second-stability}
		0\leq
		\int_W\left(
		|\nabla\varphi|^2+\frac12(\Sc_W(g_1)-\Sc(g_1)-|A_W|^2)\varphi^2
		\right)d\mathcal H_{g_1}^3
	\end{equation}
    for every 	\(\varphi\in C^\infty(W)\).
In particular, after dropping the nonpositive term $-\frac12|A_W|^2\varphi^2$, we obtain the  inequality 
\begin{equation}\label{eq:second-stability-weak} 0\leq \int_W \left( |\nabla\varphi|^2 + \frac12 \bigl( \Sc_W(g_1)-\Sc(g_1) \bigr)\varphi^2 \right) d\mathcal H_{g_1}^3. \end{equation}

    Let $\mu_1$ be the first eigenvalue of the operator \[ -\Delta_W + \frac12 \big( \Sc_W(g_1)-\sigma-\frac{\delta}{4} \big), \] and let $v>0$ be a corresponding first eigenfunction: 
	\begin{equation}\label{eq:second-eigenproblem}
		\left(-\Delta_W+\frac12(\Sc_W(g_1)-\sigma-\frac\delta 4)\right)v=\mu_1v.
	\end{equation}
	Equations \eqref{eq:first-warp-scalar} and
	\eqref{eq:second-stability} imply
	\begin{equation}\label{eq:second-eigenvalue}
		\mu_1\geq\frac{\delta}{4}\geq 0.
	\end{equation}
	On \(W\times\sph^1\), we define the warped-product metric
	\[
	g_2=g_1|_W+v^2ds^2.
	\]
	Then
	\begin{equation}\label{eq:second-warp-scalar}
		\Sc(g_2)=\Sc_W(g_1)-2\frac{\Delta_W v}{v}
		\geq\sigma+\frac{3\delta}{4}.
	\end{equation}
	
	The metric and the operator in \eqref{eq:second-eigenproblem} are invariant
	under the action of the first circle factor. Since an  eigenfunction of the first eigenvalue  is
	unique up to scaling, $v$ is also invariant. We therefore regard $u$
	and $v$ as positive functions on $Z$, and write
	\[
	g_2=g|_Z+u^2dt^2+v^2ds^2.
	\]
	A direct computation of warped product metrics shows that 
	\eqref{eq:second-warp-scalar} becomes
	\begin{equation}\label{eq:double-warp}
		\Sc(g_2)
		=\Sc_Z
		-2\frac{\Delta_Z u}{u}
		-2\frac{\Delta_Z v}{v}
		-2\langle\nabla\log u,\nabla\log v\rangle
		\geq\sigma+\frac{3\delta}{4}.
	\end{equation}
	Because \(Z\) is closed, we have
	\[
	\int_Z\frac{\Delta_Z u}{u}
	=\int_Z|\nabla\log u|^2 \textup{ and }
	\int_Z\frac{\Delta_Z v}{v}
	=\int_Z|\nabla\log v|^2.
	\]
	Integrating \eqref{eq:double-warp} over $Z$ and applying the Gauss--Bonnet formula give
	\begin{align}
		&(\sigma+3\delta/4)\operatorname{Area}_g(Z) \notag\\
		&\leq4\pi\chi(Z)
		-2\int_Z\bigl(|\nabla\log u|^2+|\nabla\log v|^2+\langle \nabla\log u,\nabla\log v\rangle\bigr)
		\,d\mathcal H_g^2 \notag\\
		&\leq4\pi\chi(Z). \label{eq:area-final}
	\end{align}
	Since the surface \(Z\) is connected and oriented,  the positivity of the left-hand side of
	\eqref{eq:area-final} forces \(\chi(Z)>0\), hence \(Z\cong\sph^2\) and
	\(\chi(Z)=2\). Since \(0\neq[Z]\in H_2(M;\mathbb Z)\),
	\[
	\operatorname{sys}_2(g)
	\leq\operatorname{Area}_g(Z)
	\leq\frac{8\pi}{\sigma+3\delta/4}
	<\frac{8\pi}{\sigma}.
	\]
    This completes the proof of Theorem \ref{thm:sysStrict}.

	\bibliographystyle{abbrv}
	\bibliography{refsystole}
\end{document}